\documentclass[11 pt]{amsart}
\usepackage{latexsym, amsmath, amssymb, longtable, booktabs,amscd,microtype,booktabs,cases}
\usepackage[square]{natbib}
\usepackage[english]{babel}
\numberwithin{equation}{section}
\usepackage[utf8x]{inputenc}
\bibpunct{[}{]}{,}{n}{}{;}
\usepackage[english]{babel}
\usepackage[utf8x]{inputenc}
\usepackage{graphicx}
\usepackage{cleveref}
\usepackage{xcolor}
\usepackage[OT2,T1]{fontenc}
\DeclareSymbolFont{cyrletters}{OT2}{wncyr}{m}{n}
\DeclareMathSymbol{\Sha}{\mathalpha}{cyrletters}{"58}
\numberwithin{equation}{section}
\usepackage[switch,columnwise]{lineno}
\usepackage{lipsum} 
\begin{document}
\newcommand\A{\mathbb{A}}
\newcommand\C{\mathbb{C}}
\newcommand\G{\mathbb{G}}
\newcommand\N{\mathbb{N}}
\newcommand\T{\mathbb{T}}
\newcommand\sO{\mathcal{O}}
\newcommand\sE{{\mathcal{E}}}
\newcommand\tE{{\mathbb{E}}}
\newcommand\sF{{\mathcal{F}}}
\newcommand\sG{{\mathcal{G}}}
\newcommand\sH{{\mathcal{H}}}
\newcommand\sN{{\mathcal{N}}}
\newcommand\GL{{\mathrm{GL}}}
\newcommand\HH{{\mathrm{H}}}
\newcommand\mM{{\mathrm{M}}}
\newcommand\fS{\mathfrak{S}}
\newcommand\fP{\mathfrak{P}}
\newcommand\fQ{\mathfrak{Q}}
\newcommand\Qbar{{\bar{\Q}}}
\newcommand\sQ{{\mathcal{Q}}}
\newcommand\sP{{\mathbb{P}}}
\newcommand{\Q}{\mathbb{Q}}
\newcommand{\tH}{\mathbb{H}}
\newcommand{\Z}{\mathbb{Z}}
\newcommand{\R}{\mathbb{R}}
\newcommand{\F}{\mathbb{F}}
\newcommand\cP{\mathcal{P}}
\newcommand\cQ{\mathcal{Q}}
\newcommand\Gal{{\mathrm {Gal}}}
\newcommand\SL{{\mathrm {SL}}}
\newcommand\Hom{{\mathrm {Hom}}}
\newtheorem{thm}{Theorem}[section]
\newtheorem{ack}[thm]{Acknowledgement}
\newtheorem{theorem}[thm]{Theorem}
\newtheorem{cor}[thm]{Corollary}
\newtheorem{conj}[thm]{Conjecture}
\newtheorem{prop}[thm]{Proposition}
\newtheorem{lemma}[thm]{Lemma}
\theoremstyle{definition}
\newtheorem{definition}[thm]{Definition}
\newtheorem{remark}[thm]{Remark}
\newtheorem{example}[thm]{Example}
\newtheorem{claim}[thm]{Claim}	
\theoremstyle{remark}
\newtheorem*{fact}{Fact}
\author{ S. Krishnamoorthy}
\address{ S. Krishnamoorthy
\newline
INDIAN INSTITUTE OF SCIENCE EDUCATION AND RESEARCH, THIRUVANANTHAPURAM, INDIA.}
\email{srilakshmi@iisertvm.ac.in}
\author{R. Muneeswaran }
\address{ R. Muneeswaran
\newline
INDIAN INSTITUTE OF SCIENCE EDUCATION AND RESEARCH, THIRUVANANTHAPURAM, INDIA.}
\email{muneeswaran20@iisertvm.ac.in}
\title{The divisibility of the class number of the imaginary quadratic fields  $\mathbb{Q}(\sqrt{1-2m^k})$}
\maketitle
\begin{abstract}
Let $h_{(m,k)}$ be the class number of $\mathbb{Q}(\sqrt{1-2m^k}).$
We prove that for any odd natural number $k,$ there exists $m_0$ such that $k \mid h_{(m,k)}$ 
for all odd $m > m_0.$ We also prove that for any odd $m \geq 3,$ $k \mid h_{(m,k)}$ (when 
$k$ and $1-2m^k$ square-free numbers) and  $p \mid h_{(m,p)}$ (except finitely many primes $p$). We deduce that for any pair of twin primes $p_1,p_2=p_1+2$, 
$p_1 \mid h_{(m,p_1)}$ or $p_2 \mid h_{(m,p_2)}.$
For any odd natural number $k$, we construct an infinite family of pairs of imaginary quadratic fields $\Q(\sqrt{d}), \Q(\sqrt{d+1})$ whose class numbers are divisible by $k$, which settles a generalized version of Iizuka's conjecture (cf : Conjecture [2.2]) for the case $n=1$.
\end{abstract}
\section{Introduction}
Let $K$ be a number field. The ideal class group $Cl_K$ is defined to be the quotient group $J_K/P_K$, where $J_K$ is the group of fractional ideals of $K$ and $P_K$ is the group of principal fractional ideals of $K$.The ideal class group $Cl_K$ is finite. The class number $h_K$ of a number field $K$ is the order of the ideal class group $Cl_K$. For any integer $n>1$, the Cohen-Lenstra heuristics \cite{cohen1984heuristics} predicts that the proportion of imaginary quadratic fields with class numbers divisible by $n$ is positive. Numerous authors have proved that for any natural number $n$, there are infinitely many quadratic fields whose class numbers are divisible by $n$.
(cf. \cite{MR85301,MR3069394,MR266898,MR335471}).
The exact structure of the class group can be studied by analyzing the divisibility properties of the class numbers.

The Birch Swinnerton-Dyer conjecture serves as an elliptic curve counterpart to the analytic class number formula. For any elliptic curve defined over $\mathbb{Q}$ with a rank of zero and a square-free conductor $N$, if a prime $p$ divides the order of $E(\mathbb{Q})$, and certain conditions on the Shafarevich-Tate group $ {\Sha}_d$ , the first author \cite{MR3529567} demonstrated that $p$ divides $|\Sha_d|$ if and only if $p$ divides the class number $h_K$ of the number field $K=\mathbb{Q}(\sqrt{-d})$.

A. Hoque, in his work \cite{MR4270672}, proved that, under specific conditions, the class number of the field $\mathbb{Q}(\sqrt{a^2-4p^n})$ is divisible by $n$. Numerous researchers have also investigated the class number divisibility for fields of the form $\mathbb{Q}(\sqrt{1-\mu^2m^n}),\ \mu\in\{1,2,\sqrt{2}\}$. (cf. \cite{louboutin2009divisibility,krishnamoorthy2021note,hoque2017divisibility,murty1999exponents,MR2843095}).

For the case $\mu=2$, B.H. Gross and D.E  Rohrlich \cite{gross} proved that for any odd integer $n>3$, there are infinitely many imaginary quadratic fields $\Q(\sqrt{1-4U^n})$ whose class numbers are divisible by $n$. In a notable work by S. Louboutin \cite{louboutin2009divisibility}, it was proved that for any integer $U \geq 2$ and an odd integer $n > 1$, the ideal class groups of imaginary quadratic fields, $\mathbb{Q}(\sqrt{1 - 4U^n})$, has an element of order $n$. 

For the case $\mu=1$, Murty \cite{RM99} proved that the class number of $\Q(\sqrt{1-U^n})$ is divisible by $n$,  if $1-U^n$ is square-free. A. Hoque \cite{MR4430107} proved that the class number of $\mathbb{Q}(\sqrt{1-V^n})$ is divisible by $n$ for odd values of $n\geq3$ and $V\geq3$, except for the case $(n,V)=(5,3)$. For the case $\mu=\sqrt{2}$, consider the imaginary quadratic field $\mathbb{Q}(\sqrt{1-2m^k})$, where $h_{(m,k)}$ denotes its class number. K. Chakraborty and A. Hoque \cite{hoque2017divisibility} proved that for any odd integer $m \geq 3$, the class number $h_{(m,3)}$ is divisible by 3.
The first author with S. Pasupulati \cite{krishnamoorthy2021note}, generalized the above results and established that, for given odd primes $p$, $q$, and a natural number $r$, the class number $h_{(m,p)}$ is divisible by $p$ when $m=q^r$. This observation implies that, for a fixed prime $p$, there exists an infinite family of imaginary quadratic fields within this specified parameter range. 

Observing the results above, it is apparent that while the cases of $\mu=1$ and $\mu=2$ have been extensively explored, the scenario for $\mu=\sqrt{2}$ is relatively less explored. Hence, in this paper, we are exploring the cases of $\mu=\sqrt{2}$.

There are two different ways we can try to get the results for the class number divisibility of the family $\Q(\sqrt{1-2m^k})$.
The first method is by fixing the value of $k$, exploring the possible values of $m$. The second method is by fixing $m$, exploring the possible values of $k$. Let us start with our first method, that is by fixing $k$, we explore the possible values of $m$. 
The class number of $\Q(\sqrt{1-2(4)^3})$ is $5$, which is not divisible by $3$. Through sage computations, when $k=3,$ we observe that out of first $250$ even values of $m$, only for $106$ even values, the class numbers are divisible by $3$. Hence throughout this paper, we consider only odd values of $m$.
\begin{remark}
There are several results on indivisibility of class numbers also there. For an example, Gauss proved that class number of $\Q(\sqrt{-p}), \ p\equiv3\pmod{4}$ is odd. From this, we get infinitely many quadratic fields whose class numbers are not divisible by $2$.  Hartung \cite{MR352040} proved that there exists an infinite family of imaginary quadratic fields whose class numbers are not divisible by $3$. Due to the observations made on the even values of $m$, one can think about the condition on even values of $m$, such that $h_{(m,k)}$ is not divisible by $k$.
\end{remark}
 By our first method of Approach, by fixing $k=p^r$, where $p$ is an odd prime number, we get the following \Cref{t5}.
\begin{theorem}\label{t5}
For an odd prime number $p$ and any natural number $r$, let $m$ be an odd integer greater than $2^{\frac{p-2}{p^{r-1}}}$ and $k = p^r$. Then $k$ divides the class number $h_{(m,k)}$.
In particular, if $r > 1$, then $k$ divides $h_{(m,k)}$ for all odd $m>1$.		
\end{theorem}

By using \Cref{t5}, we get the following result, which works for any odd number $k>1$.
\begin{cor}\label{c6} 
Let $k\geq3$ be an odd number with prime factorization $k=p_1^{k_1}p_2^{k_2}...p_n^{k_n}.$ For any odd integer $m>\max\{2^{\frac{p_i-2}{p_i^{k_i-1}}}: 1\leq i\leq n\},$ we get $k\mid h_{(m,k)}.$ 
\end{cor}

For any odd number $k$, Xie and Chao \cite{MR4173455}, and A. Hoque \cite{MR4430107} proved a particular case of Iizuka's conjecture $(n=1)$ by producing infinitely many $d$ such that class numbers of $\Q(\sqrt{d})$ and $\Q(\sqrt{d+1})$ are simultanously divisible by $k$. By using Corollary\ref{c6}, we also produced infinitely many $d$, with such property in Corollary\ref{CC}. However our collection of fields are different.

\begin{cor}\label{CC}
For any odd number $k>1$, there exist infinitely many integers $d$ such that the class numbers of $\mathbb{Q}(\sqrt{d})$ and $\mathbb{Q}(\sqrt{d+1})$ are divisible by $k$.
\end{cor}
We get the following result for the class numbers of bi-quadratic fields by using Corollary\ref{c6}.
\begin{cor}\label{biquad}
For any odd integer $k\geq3$, there exist infinitely many imaginary bi-quadratic fields whose class numbers are divisible by $k$.
\end{cor}
By approaching the divisibility problem through the second method, that is by fixing $m$, finding possible values for $k$, we get $k$ is in the set of prime numbers (\Cref{t4}).
\begin{theorem}\label{t4}
Let $m\geq3$ be an odd integer and $p$ be an odd prime number.\\
(a) For all but finitely many primes $p,$ $p\mid h_{(m,p)}.$
For a fixed $m$, the collection of fields  $\mathbb{Q}(\sqrt{1-2m^p})$ such that $p\mid h_{(m,p)}$ is infinite.\\
(b)  If $1-2m^p$ is square-free, then $p\mid h_{(m,p)}.$\\
(c) If $p_1, \ p_2$ be any pair of twin primes, then $p_1\mid h_{(m,p_1)}$ or $p_2\mid h_{(m,p_2)}.$
\end{theorem}
We get the square-free part of $k$ divisibility by using the above theorem.
 \begin{cor}\label{c7} 
Let $m\geq3$ be an odd integer and $p$ be an odd prime number.\\
(a) There exist a natural number $r$ such that for any odd number $k\geq3$ which is co-prime to $r,$ the square-free part of $k$ divides $h_{(m,k)}.$\\
(b) If $t\geq3$ is an odd integer such that $1-2m^k$ is a square-free integer, then the square-free part of $k$ divides $h_{(m,k)}.$
\end{cor}
 \begin{remark}
Let $p_1, \ p_2=p_1+2$ be a pair of twin primes. By using \Cref{t5}, if we choose odd $m>2^{p_1}$, then $p_i\mid h_{(m,p_i)}$ for $i=1,2$, which is stronger than of \Cref{t4} (c) part.
 \end{remark}
 To prove the main theorems, we begin by demonstrating that $\pm2^{\frac{k-1}{2}}(1+\sqrt{1-2m^k})$ is not a $p^{th}$ power in the ring of integers of $\mathbb{Q}(\sqrt{1-2m^k})$ (see Proposition \ref{p16}). By using this result, we then construct an element of order $p^r$ within the class group of $\mathbb{Q}(\sqrt{1-2m^k})$, where $k=p^r$ (refer to Proposition \ref{p19}).

We introduce certain preliminary concepts in Section\ref{p}, and in Section \ref{PT5}, we prove Theorem \ref{t5}. In Section \ref{PT4}, we provide the proof of Theorem \ref{t4}.
\section{Preliminaries}\label{p}
In a noteworthy work, Iizuka \cite{MR3724158} demonstrated the existence of an infinite family of pairs of imaginary quadratic fields $\mathbb{Q}(\sqrt{d})$ and $\mathbb{Q}(\sqrt{d+1})$ with $d \in \mathbb{Z}$, where the class numbers of both fields are simultaneously divisible by $3$.
Based on this result and also on some numerial evidences, Iizuka conjectured the following.
\begin{conj}[Iizuka]\label{conjecture}
For any prime $p$ and any positive integer $n,$ there is an infinite family of $n+1$ successive imaginary (or real) quadratic fields 
$$\Q(\sqrt{d}), \Q(\sqrt{d+1}), \cdots, \Q(\sqrt{d+n})$$ with $d \in \Z$ whose class numbers are divisible by $p.$
\end{conj}
A more comprehensive form of the aforementioned conjecture is stated below.
\begin{conj}[Iizuka]\label{conjecture1}
For any odd number $k$ and any positive integer $n,$ there is an infinite family of $n+1$ successive imaginary (or real) quadratic fields 
$$\Q(\sqrt{d}), \Q(\sqrt{d+1}), \cdots, \Q(\sqrt{d+n})$$ with $d \in \Z$ whose class numbers are divisible by $k.$
\end{conj}
A less stringent form of Iizuka's conjecture would be instead of considering consecutive numbers $d$, $d+1$, ..., $d+n$, we replace them with the numbers in an arithmetic progression ${d+ib}$ for a fixed integer $b$ and $i$ ranging from $1$ to $n$. More broadly, the class numbers of $\Q(\sqrt{d+p(i)})$ are divisible by $k$ for $i=1$ to $n$, where $p(x)\in\Z[x]$.
 \begin{conj}[Weaker version]\label{conjecture11}
For any odd number $k$ and any positive integer $n,$ there is an infinite family of $n+1$  imaginary (or real) quadratic fields 
$$\Q(\sqrt{d}), \Q(\sqrt{d+b}), \Q(\sqrt{d+2b}),\cdots, \Q(\sqrt{d+nb})$$ with $d \in \Z$ whose class numbers are divisible by $k.$
\end{conj}
 
Let K be a number field and Let S be a finite set of valuations containing all the archimedean valuations. then $R_S=\{x\in K:V(x)\geq 0 , \forall v\notin S\}$ is called the set of S-integers.\\
If $K=\Q$ and $S$ = set of all archimedean valuations in $\Q$, then $R_S=\Z$
\begin{lemma}{Siegel's Theorem (\cite{silverman2009arithmetic}, Chapter IX, Theorem 4.3).}\label{siegel}
Let K be a number field and S be a finite set of valuations containing all the archimedean valuations on K. Let $f(x)\in K[x]$ be a polynomial of degree $d\geq 3$ with distinct roots in the algebraic closure $\Bar{K}$ of K. Then the equation $y^2=f(x)$ has finitely many solutions in S-integers $x,y\in R_S$.
\end{lemma}
\begin{lemma}\label{c12}
For a fixed $k$, the collection of number fields $\mathbb{Q}(\sqrt{1-2m^k}),$ where $m\geq3$ is any odd positive integer is an infinite collection.
\end{lemma}
\begin{proof}
Consider the polynomial $f(x)=\frac{1-2x^k}{d_0},$ $d_0$ is any square-free integer. It has distinct roots in the algebraic closure of $\mathbb{Q}$.
Hence  by \Cref{siegel}, $z^2=\frac{1-2x^k}{d_0}$ has only finitely many integral solutions $(x,z)\in\Z\times\Z$.
Therefore $\{\mathbb{Q}(\sqrt{1-2m^k})  \ :  m \ \mathrm{ \geq 3} \ \mathrm{is} \
\mathrm{odd} \ \mathrm{positive} \ \mathrm{integer} \}$ is an infinite set.
\end{proof}
 We present a result by Yann Bugeaud and T.N. Shorey concerning a Diophantine equation. Consider coprime positive integers $D_1$ and $D_2$, let $D = D_1D_2$, and $m \geq 2$ be an integer coprime with $D$. Let $\lambda \in \{1, \sqrt{2}, 2\}$, where $\lambda = 2$ when $m$ is even. The equation of interest is $D_1x^2 + D_2 = \lambda^2 m^k$.

We denote $F_i$ as the Fibonacci sequence defined by $F_0=0$, $F_1=1$, and $F_i=F_{i-1}+F_{i-2}$, and $L_i$ as the Lucas sequence defined by $L_0=2$, $L_1=1$, and $L_i=L_{i-1}+L_{i-2}$ for $i\geq 2$. \\
We define subsets $F$, $G$, and $H$ of $\mathbb{N}\times\mathbb{N}\times\mathbb{N}$ as follows:
$$F=\{(F_{i-2\epsilon},L_{i+\epsilon},F_i):i\geq 2,\epsilon\in\{\pm1\}\},$$
$$G=\{(1,4m^r-1,m):m\geq 2, r\geq 1\},$$
$$H=\{(D_1,D_2,m):\exists \ r,s\in\mathbb{N}\text{ such that }D_1s^2+D_2=\lambda^2m^r \text{ and } 3D_1s^2-D_2=\pm\lambda^2\}.$$
Let $N(\lambda, D_1, D_2, m)$ be defined as the count of pairs $(x, k) \in \mathbb{Z}^+ \times \mathbb{Z}^+$ satisfying the equation $D_1x^2 + D_2 = \lambda^2m^k$. Let $S = \{(2,13,3,4),(\sqrt{2},7,11,9),(\sqrt{2},1,1,5),(\sqrt{2},1,1,13),(2,1,3,7), \\(1,1,19,55),(1,1,341,377),(1,2,1,3),(2,7,1,2)\}$. 
\begin{theorem}\label{t13} (Corollary 1, Theorem 2 of \cite{bugeaud2001number}). Let $m \geq 2$ be an integer.\\
If $(\lambda, D_1, D_2, m) \notin S$ and $(D_1, D_2,m) \notin F \cup G \cup H$, then $N(\lambda,D_1,D_2,m) \leq 2^{\omega(m)-1} $,
where $\omega(m)$ is the number of distinct prime divisors of $m$.
If $(\lambda, D_1, D_2, m) \in S$, then $N(\lambda,D_1,D_2,m) = 2 $.
\end{theorem}
\begin{lemma}\label{l15}
Let $m \geq 3$ be an odd number.\\
(i) For any positive integer $D$ and $\mathrm{gcd}(D,m)=1$, $N(\sqrt{2},D,1,m)$ is finite.\\
(ii) $N(\sqrt{2}, 2m-1,1,m) \leq 2^{\omega(m)-1} $.
	\end{lemma}
	\begin{proof}
(i) Let $D_1=D,D_2=1,\lambda=\sqrt{2}$. From \Cref{t13}, the Diophantine equation $Dx^2+1=2m^k$ has a finite number of solutions if $(D_1,D_2,m)=(D,1,m)\notin F\cup G\cup H$. If $(D,1,m)\in F$, then $(D,1,m)=(F_{i-2\epsilon},L_{i+\epsilon},F_i)$ for some $i,\epsilon$. By comparing the second coordinates, we get $L_{i+\epsilon}=1.$ Hence $i=2 \text{ and } \epsilon=-1$, which implies that $F_i=1.$ By comparing the third coordinates, we get $m=F_i=1$, which is a contradiction to $m\geq 3$. If $(D,1,m)\in G$, then $(D,1,m)=(1,4m^r-1,m).$
Hence $4m^r-1=1$, which is a contradiction to $m\in\Z$.
If $(D,1,m)\in H$, then $\exists$ $r,s\in\mathbb{N} \ni Ds^2+1=2m^r,$   $3Ds^2-1=\pm2.$
Suppose $3Ds^2-1=2$, then $Ds^2=1.$ Hence $2m^r-1=1.$ Thus $m=1,$ which is not possible. Suppose $3Ds^2-1=-2,$ then $D=\frac{-1}{3s^2},$ which is a contradiction to $D\in\Z.$ Thus $(D,1,m)\notin F\cup G\cup H$.
Hence (i) follows.\\
(ii) Consider $D_1=2m-1,$ $D_2=1, \ \lambda=\sqrt{2}$. We see that 
$(\lambda,D_1,D_2,m) = (\sqrt{2}, 2m-1,1,m) \notin S$.
Clearly $\mathrm{gcd}(D_1,D_2)=\mathrm{gcd}(2m-1,1)=1, \  \mathrm{gcd}(D_1D_2,m)=\mathrm{gcd}(2m-1,m)=1$ and $m\geq3$.\\
Proceeding as in the proof of  (i), $(2m-1,1,m)\notin F\cup G\cup H$.
Hence  $N(\sqrt{2}, 2m-1,1,m) \leq 2^{\omega(m)-1} $ by \Cref{t13}. \end{proof}
  \section{Proof of Theorem \ref{t5}}\label{PT5}
  In this section we mainly focus of the proof of the \Cref{t5} and it's corollaries.
\begin{lemma}\label{division Lemma}
For any integer $m>2^{\frac{p-2}{p^{r-1}}}$ and for any natural number $r,$ $2m^{p^{r-1}}-1$ does not divide $m^{p^{r}-p^{r-1}}-1.$
  \end{lemma}
\begin{proof}
Let $m>2^{\frac{p-2}{p^{r-1}}}$, that is $2m^{p^{r-1}}-1>2^{p-1}-1.$ \\
If $2m^{p^{r-1}}-1$ divides $m^{p^{r}-{p^{r-1}}}-1,$ 
then we have,
$$(2m^{p^{r-1}}-1)\mid((m^{p^{r}-p^{r-1}}-1)-(2m^{p^{r-1}}-1)), \  \mathrm{i.e.,} \ (2m^{p^{r-1}}-1)\mid m^{p^{r-1}}(m^{p^{r}-2p^{r-1}}-2).$$ 
Since $(2m^{p^{r-1}}-1,m^{p^{r-1}})=1$, we get, $$(2m^{p^{r-1}}-1)\mid (m^{p^{r}-2p^{r-1}}-2).$$
We also have,
$$(2m^{p^{r-1}}-1)\mid(m^{p^r-2p^{r-1}}-2)-2(2m^{p^{r-1}}-1)).$$ 
Hence $(2m^{p^{r-1}}-1)\mid(m^{p^r-3p^{r-1}}-2^2).$ Repeating the above process, we get, $$(2m^{p^{r-1}}-1)\mid(m^{p^r-np^{r-1}}-2^{n-1}), \ 1\leq n\leq p.$$ In particular for $n=p$, we get $(2m^{p^{r-1}}-1)\mid (2^{p-1}-1),$  which is a contradiction to\\ $2m^{p^{r-1}}-1>2^{p-1}-1.$ Hence the result holds.
\end{proof}
\begin{prop}\label{p16}
(a) Let $m>2^{\frac{p-2}{p^{r-1}}}$ be an odd integer and $k=p^r$, where $p$ is an odd prime number and $r$ is any natural number.
Let $\alpha=1+\sqrt{1-2m^{k}}.$ Then $\pm2^{\frac{k-1}{2}}\alpha$ is not a $p^{th}$ power in the ring of integers of $\Q(\sqrt{1-2m^k}).$\\
(b) Let $m\geq3$ be an odd integer and $p$ be an odd prime number. Let $\alpha=1+\sqrt{1-2m^p}.$ For all but finitely many odd primes $p$, $\pm2^{\frac{p-1}{2}}\alpha$ is not a $p^{\mathrm{th}}$ power in the ring of integers of $\mathbb{Q}( \sqrt{1-2m^p} )$.\\
		(c) Let $m\geq3$ be an odd integer and $p$ be an odd prime number. Let $\alpha=1+\sqrt{1-2m^p}.$ If $1-2m^p$ is square-free, then $\pm2^{\frac{p-1}{2}}\alpha$ is not a $p^{th}$ power in the ring of integers of $\mathbb{Q}(\sqrt{1-2m^p}).$\\
		(d) Let $m\geq3$ be an odd integer and $p_1, \ p_2$ be any pair of twin primes. Then at least for one of the $p_i$, $\pm2^{\frac{p_i-1}{2}}(1+\sqrt{1-2m^{p_i}})$ is not a $p_i^{th}$ power in the ring of integers of $\mathbb{Q}(\sqrt{1-2m^{p_i}}).$
	\end{prop}
	\begin{proof}
(a): Let $K=\Q\left(\sqrt{1-2m^k}\right)$. Let $\mathbb{Z}_K$ be the ring of integers of $K$. To establish the claim, it is sufficient to prove it for $2^{\frac{k-1}{2}}\alpha$. Let $1-2m^k=n^2d,$ where $n\geq1$ and $d$ is square-free. Since $m$ is an odd positive integer, it follows that $d\equiv3\pmod4$. From this we get, $\Z_K=\Z+\Z\sqrt{d}$.

Assume that $2^{\frac{k-1}{2}}\alpha=\beta^p$ is a $p^{th}$ power in $\Z_K$, where $\beta=a+b\sqrt{d}\in\Z_K$. Then $\beta^2-2a\beta+N=0$, where $N=a^2-db^2.$ Hence $\beta^{t+2}-2a\beta^{t+1}+N\beta^t=0$ for $t\geq1$. Comparing the real and imaginary parts on the both sides of $\beta^{t+2}-2a\beta^{t+1}+N\beta^t=0$, we get, $a_{t+2}-2aa_{t+1}+Na_t=0$ and $b_{t+2}-2ab_{t+1}+Nb_t=0$, for $t\geq1$, where $\beta^t=a_t+b_t\sqrt{d}.$ It follows from induction on $t$ that $a$ divides $a_t$ for any odd number $t\geq1$. Since $b_1=b$ and $b_2=2ab$, it follows from induction that $b$ divides $b_t$ for $t\geq1$. In Particular, $a_p=2^{\frac{k-1}{2}}$ and $b_p=2^{\frac{k-1}{2}}n$, for $t=p$, we get $a\mid2^{\frac{k-1}{2}}$ and $b\mid2^{\frac{k-1}{2}}n.$ Moreover, taking norm in the equality $2^{\frac{k-1}{2}}\alpha=\beta^p,$ that is, $N\left(2^{\frac{k-1}{2}}\left(1+\sqrt{1-2m^{k}}\right)\right)=N\left(\left(a+b\sqrt{d}\right)^p\right)$, we get $2^km^k=(a^2-b^2d)^p$, That is $2^{p^r}m^{p^r}=(a^2-b^2d)^p$. Hence we get $2^{p^{r-1}}m^{p^{r-1}}=a^2-db^2.$ Hence $a$ and $b$ both odd or both even.\\
\textbf{Case 1 (Both $a$ and $b$ odd):}\\
If $r>1$, then taking  $2^{p^{r-1}}m^{p^{r-1}}=a^2-db^2$ modulo 4, we get $0\equiv 2\pmod4$, which is a contradiction. Suppose $r=1.$
We have  $a\mid2^{\frac{p-1}{2}}$ and $b\mid2^{\frac{p-1}{2}}n,$ which implies $a=\pm1$ and $b\mid n.$ The equation $2^{p^{r-1}}m^{p^{r-1}}=1-db^2$ becomes $2m-1=-db^2$. Thus $(2m-1)\mid(2m^p-1)$ and hence we have $(2m-1)\mid((2m^p-1)-(2m-1)),$ which is $(2m-1)\mid2m(m^{p-1}-1).$ We know that $(2m-1,2m)=1$. Hence $(2m-1)\mid (m^{p-1}-1),$ which is a contradiction to \Cref{division Lemma}.\\
\textbf{Case 2 (Both $a$ and $b$ even):}\\
We have  $a\mid2^{\frac{k-1}{2}}$ and $b\mid2^{\frac{k-1}{2}}n.$ Hence $a=\pm2^{s}$ and $b=2^tx$, $x$ is odd which divides $n$.\\
\textbf{Subcase 1 $(s>t)$:}\\
The equation $2^{p^{r-1}}m^{p^{r-1}}=a^2-db^2$ becomes $2^{p^{r-1}}m^{p^{r-1}}=2^{2t}(2^{2(s-t)}-x^2d).$
Since $s>t$ and $x,d$ are odd numbers, the maximum power of 2 dividing the right side of the above equation is $2t$. But the maximum power of 2 dividing the left hand side is an odd number, which is not possible.\\
\textbf{Subcase 2 $(t>s)$:}\\
The equation $2^{p^{r-1}}m^{p^{r-1}}=a^2-db^2$ becomes $2^{p^{r-1}}m^{p^{r-1}}=2^{2s}(1-2^{2(t-s)}x^2d).$ Since $t>s$ , the maximum power of 2 dividing right side of the above equation is $2s$. But maximum power of 2 dividing left hand side is an odd number, which is not possible.\\
\textbf{Subcase 3 $(t=s)$:}\\
The equation $2^{p^{r-1}}m^{p^{r-1}}=a^2-db^2$ becomes $2^{p^{r-1}}m^{p^{r-1}}=2^{2s}(1-x^2d).$\\
Clearly $1-x^2d\equiv2\pmod4.$ Hence $2s=p^{r-1}-1.$ Thus $2m^{p^{r-1}}=1-x^2d$, which is $2m^{p^{r-1}}-1=-x^2d.$ We know that $x$ divides $n$ and $-n^2d=2m^{p^r}-1$. Hence $(2m^{p^{r-1}}-1)\mid (2m^{p^r}-1)$. From this we get, $(2m^{p^{r-1}}-1)\mid ((2m^{p^r}-1)-(2m^{p^{r-1}}-1))$, that is $(2m^{p^{r-1}}-1)\mid 2m^{p^{r-1}}(m^{p^r-p^{r-1}}-1).$ Clearly $(2m^{p^{r-1}}-1,2m^{p^{r-1}})=1.$ 
Hence $(2m^{p^{r-1}}-1)\mid (m^{p^r-p^{r-1}}-1).$ By \cref{division Lemma}, this is not possible. Hence the proof follows.\\
(b) : If $\pm2^{\frac{p-1}{2}}\alpha$ is a $p^{th}$ power, then as in the proof of (a), we get $2m=a^2-db^2$ and $a\mid2^{\frac{p-1}{2}}$, $b\mid2^{\frac{p-1}{2}}n$. Taking modulo over 4 on both sides, we get $a^2+b^2\equiv2\pmod4.$ Thus $a$ and $b$ both odd. Hence $a=\pm1$, $b\mid n$ and $2m=1-db^2$, which we write $(2m-1)(\frac{n}{b})^2+1=2m^p$. It follows from Lemma \ref{l15} (ii), $p$ takes at most $2^{\omega(m)-1}$ values  and the desired result follows.\\
(c): If $1-2m^p$ is square-free, then $n=1$ and $1-2m^p=d$.  If $\pm2^{\frac{p-1}{2}}\alpha$ is a $p^{th}$ power in the ring of integers of $\mathbb{Q}(\sqrt{1-2m^p}),$ then as in the proof of part (b), we get, $b\mid n$ and $2m=1-db^2.$ Hence $b=\pm1$ and $1-2m=d$. Equating the values of $d$, we get $p=1$, which is a contradiction. 
(d) :  Let $p_1, \ p_2$ be any pair of twin primes such that $p_2-p_1=2$. Assuming the contrary, as in the proof of (b), we get $(2m-1)(\frac{n}{b})^2+1=2m^{p_i}$ for $i=1,2$. Hence, $(2m-1)\mid (2m^{p_i}-1)$ for $i=1,2$. Thus $(2m-1)\mid2m^{p_1}(m+1)(m-1).$ Since $2m-1=2(m-1)+1$ is co-prime to $m-1$ and also co-prime to $2m^{p_1}$, we get $(2m-1)\mid(m+1),$ which implies that $2m-1\leq m+1$. Thus $m\leq2$, which is a contradiction to $m\geq3$. Hence the result follows.

	\end{proof}

 \begin{prop}\label{p19}
Let $m\geq3$ be an odd integer, $p$ be an odd prime and $r$ be a natural number. Let $\alpha=1+\sqrt{1-2m^k},$ where $k=p^r.$ 
If $\pm2^{\frac{k-1}{2}}\alpha$ is not a $p^{th}$ power in ring of integers of $\Q(\sqrt{1-2m^k})$, then $k\mid h_{(m,k)}.$
\end{prop}
\begin{proof}
Let $K=\Q\left(\sqrt{1-2m^k}\right).$
Write $1-2m^k=n^2d$ , $n\geq1$ and $d$ square-free.  Since $m$ is odd positive integer, we get $d\equiv3\pmod4$. 
Now, $N_{K/\mathbb{Q}}(\alpha)=2m^k$ and any prime dividing $m$ splits in $K$. Since $d\equiv 3\pmod4$ the ideal $\left\langle 2 \right\rangle$ is ramified. Let $\left\langle2\right\rangle=P^2$ for some prime ideal $P$. Let $m=p_1^{r_1}p_2^{r_2}...p_n^{r_n}$, each prime $p_i$ is odd and pairwise co-prime, be the prime factorization of $m$. Since $N(\left\langle\alpha\right\rangle)=N_{K/\mathbb{Q}}(\alpha)=2p_1^{r_1k}p_2^{r_2k}...p_n^{r_nk}$, the prime ideal decomposition of $\left<\alpha\right>$ must have prime ideals whose norms are $2,p_1,p_2,...,p_r$. Suppose the prime decomposition of $\left\langle\alpha\right\rangle$ has both prime ideals  $P_i$ and $P_i'$ as factors, where $<p_i>=P_iP_i'$, then we get $p_i$ divides the real part of $\alpha$. But the real part of $\alpha$ is $1$. Thus the factors of the prime decomposition of $\left\langle\alpha\right\rangle$ must have exactly one of prime ideal lies above $p_i$, for each $p_i$. Hence the prime decomposition of $\left\langle\alpha\right\rangle$ is given by $\left\langle\alpha\right\rangle=PP_1^{t_1}P_2^{t_2}...P_n^{t_n}$, where each $P_i$ is a prime ideal lies above $p_i$ and each $t_i$ is a positive integer. Since $p_i$ splits over $K$, we have $N(P_i)=p_i$ for all $i$. Hence we get $t_i=r_ik$.
Consider the ideal $I=PP_1^{r_1}P_2^{r_2}...P_n^{r_n}$ of $\mathcal{O}_K$.
We have,
\begin{align*}
I^k=P^kP_1^{t_1}P_2^{t_2}...P_n^{t_n}=\left\langle 2\right\rangle^{\frac{k-1}{2}}PP_1^{t_1}P_2^{t_2}...P_n^{t_n}=\left\langle    2\right\rangle^{\frac{k-1}{2}}\left\langle\alpha\right\rangle=\left\langle 2^{\frac{k-1}{2}}\alpha \right\rangle.
\end{align*}
We claim that the order of  $I$ is $p^r$. If not, then the order of $I$ must be $p^h$, where $0\leq h\leq r-1$. Thus the order of $I$ divides $p^{r-1}$. Hence $I^{p^{r-1}}=\left\langle\beta\right\rangle$ for some $\beta\in \mathcal{O}_K.$ Thus $I^{p^r}=\left\langle\beta^p\right\rangle=\left\langle 2^{\frac{k-1}{2}}\alpha\right\rangle$. 
Thus $\beta^p= u 2^{\frac{k-1}{2}}\alpha$ for some unit $u$ in $\mathcal{O}_K$. Since $1-2m^k<-3$, the only units of $\mathcal{O}_K$ are 1 and -1. Hence $\beta^p= \pm 2^{\frac{p-1}{2}}\alpha$, which is a contradiction to $\pm2^{\frac{p-1}{2}}\alpha$ is not a $p^{th}$ power in $\mathcal{O}_K$. Hence the order of $I$ is $p^r$. Hence $k\mid h_{(m,k)}.$
\end{proof}
\hspace{-1.1cm}{\it \textbf{Proof of \Cref{t5}}}. Let $m>2^{\frac{p-2}{p^{r-1}}}$  be an odd number. Let $k=p^r$ for some odd prime $p$ and a natural number $r.$ From Proposition\ref{p16} (a) and Proposition \ref{p19} we observe that $k\mid h_{(m,k)}.$ 
When $r>1$, we have $2>2^{\frac{p-2}{p^{r-1}}}$. Hence when $r>1,$ the result is true for any odd natural number $m>1.$\vspace{0.1cm}\\
{\it \textbf{Proof of Corollary \ref{c6}}}. 
Let $k\geq3$ be an odd number with prime factorization $k=p_1^{k_1}p_2^{k_2}...p_n^{k_n}.$ Consider any odd integer $m>\max\{2^{\frac{p_i-2}{p_i^{k_i-1}}}: 1\leq i\leq n\}.$\\
We have $\Q(\sqrt{1-2m^k})=\Q\left(\sqrt{1-2(m^{\frac{k}{p_i^{k_i}}})^{p_i^{k_i}}}\right).$ 
By the \Cref{t5}, $p_i^{k_i}\mid h_{(m,k)}.$ 
Since $p_i^{k_i}$ is arbitrary, we get $k\mid h_{(m,k)}.$
\vspace{0.1cm}\\
	{\it \textbf{Proof of Corollary \ref{CC}}}: 
	It follows from \Cref{t5}, any odd natural number $k$ divides the class number of $\Q(\sqrt{4(1-2m^k)^k})=\Q\left(\sqrt{1-2m^k}\right) \text{for any odd}  \ m>\max\{2^{\frac{p_i-2}{p_i^{k_i-1}}}: 1\leq i\leq n\}, $ where $k=p_1^{k_1}p_2^{k_2}...p_n^{k_n}.$
	Let $U=2m^k-1.$ Then
	$k$ divides the class number of $\Q\left(\sqrt{1-4U^k}\right)=\Q\left(\sqrt{1-4(2m^k-1)^k} \right)=\Q\left(\sqrt{4(1-2m^k)^k+1}\right)$ by Theorem 1, \cite{louboutin2009divisibility}. Let $d=4(1-2m^k)^k.$ Then $k$ divides class numbers of $\Q(\sqrt{d}),\Q(\sqrt{d+1}).$ The infiniteness of $d$ follows from \Cref{t4}.\\
	{\it \textbf{Proof of Corollary \ref{biquad}}}
	Fix an odd number $k=p_1^{k_1}p_2^{k_2}...p_n^{k_n}\geq3$. \\Let $S_1 =
	\{m>\max\{2^{\frac{p_i-2}{p_i^{k_i-1}}}: 1\leq i\leq n\}:$ $m$ is not a square, $m\equiv1\pmod4 $. Let $K=\{\mathbb{Q}(\sqrt{1-2m^k})$, $k \mid h_K \}$.
	Clearly this is an infinite set. For $m\in S_1$, 
	consider the bi-quadratic field $K_m=\mathbb{Q}(\sqrt{1-2m^k},\sqrt{m})$. Denote $L_m^1=\mathbb{Q}(\sqrt{1-2m^k}),L_m^2=\mathbb{Q}(\sqrt{m}),$ and $L^3_m=\mathbb{Q}(\sqrt{1-2m^k}\sqrt{m})$. Observe that $L_m^1\neq L_m^2$ because $1-2m^k\equiv3\pmod4$. 
	Since $m$ is not a square, $L_m^1,L_m^2$ and $L_m^3$ are three distinct quadratic sub fields of $K_m$. Let $h_m,h_m^1,h_m^2,h_m^3$ be the class numbers of $K_m,L_m^1,L_m^2,L_m^3$ respectively. Then by Lemma 2, \cite{cornell1991note}, we have $h_m=\frac{h_m^1h_m^2h_m^3}{2^i},i=0,1$. 
	Since $m \in S_1$, $k$ divides $h_m^1$. Since $k$ is odd, $k$ must divide $h_m$. The infiniteness of the set $\{K_m:m\in S_1\}$ follows from infiniteness of the set $S_1$ and Siegel's Theorem.
\section{Proof of Theorem \ref{t4} }\label{PT4}
In this section, we mainly focus in the proof of the \Cref{t4} and it's corollary. We state some remarks related to our theorems.\\
\hspace{-1.1cm}{\it \textbf{Proof of \Cref{t4}}}.(a) : For a given odd integer $m\geq3$, by Proposition \ref{p16} (b), for all but finitely many odd primes $p,$ $\pm2^{\frac{p-1}{2}}(1+\sqrt{1-2m^p})$ is not a $p^{th}$ power in the ring of integers of $\Q(\sqrt{1-2m^p}).$ By Proposition \ref{p19}, we get $p\mid h_{(m,p)},$
	for all but finitely many odd primes $p$.  Given positive square-free integer $D$, it can be the square-free part of $2m^p-1$  for only finitely many primes $p$ by Lemma \ref{l15} (i). Hence for a fixed $m \geq 3$ odd integer, the collection of fields $\mathbb{Q}(\sqrt{1-2m^p})$ such that $p\mid h_{(m,p)}$
	is infinite.\\
	(b) : Assume that $2m^p-1$ is square-free for an odd prime $p$ and $m\geq3$ an odd number.
By Proposition \ref{p16} (c)  and Proposition \ref{p19}, we get $p\mid h_{(m,p)}.$\\
(c) : The proof follows from Proposition \ref{p16} (d) and Proposition \ref{p19}.\\
{\it \textbf{Proof of Corollary \ref{c7}}}. 
(a) : Let $J$ be the set of finitely many exceptional primes arising in\\ \Cref{t4} (a). If $J=\emptyset$, then set $r=1$. If $J=\{p_1,p_2,...,p_n\}$, then set $r=p_1p_2...p_n$. Consider an odd integer $k\geq3,$ which is co-prime to $r$. Suppose $p$ is a prime such that $p\mid k$.  Since $\mathbb{Q}(\sqrt{1-2m^k})=\mathbb{Q}(\sqrt{1-2(m^{\frac{k}{p}})^p})$, by our choice of $r$ and by \Cref{t4} (a), $p\mid h_{(m,k)}$. Since $p$ is an arbitrary prime divisor of $k$, the square-free part of $k$ divides the class number of $\mathbb{Q}(\sqrt{1-2m^k})$.\\
(b) : Suppose $p$ is a prime which divides an odd integer $k\geq3$. Assume that $1-2m^k$ is square-free. Since $\mathbb{Q}(\sqrt{1-2m^k})=\mathbb{Q}(\sqrt{1-2(m^{\frac{k}{p}})^p})$, by \Cref{t4} (b), $p$ divides the class number of $\mathbb{Q}(\sqrt{1-2m^k})$. Since $p$ is an arbitrary prime divisor of $k$, we conclude that the square-free part of $k$ divides $h_{(m,k)}.$

\begin{remark}
Using SageMath, we get the class numbers of $\Q(\sqrt{1-2m^3})$ are divisible by $3$ only for $106$ integer values of $m$ between the range $-400$ to $-1$ . In general, it will be an interesting question to find a condition on $k$ and negative values of $m$ such that $k\mid h_{(m,k)}.$ 
\end{remark}
\begin{remark}
Twin prime conjecture states that there are infinitely many twin primes. If the twin prime conjecture is true, then for each pair of twin primes we can have at least one prime $p$ which divides class number of $\mathbb{Q}(\sqrt{1-2m^p})$. This shows that there are infinitely many primes $p$ which divides the class number of $\mathbb{Q}(\sqrt{1-2m^p})$.
\end{remark}

 \begin{remark}
  For a given prime number p, several authors constructed infinitely many imaginary quadratic fields whose class numbers are divisible by p. However positive proportion of such fields are not yet known for even a specific prime $p$. We intend to attempt this in our future work.
 \end{remark}

\hspace{-1.35cm}Table 1 : Examples which satisfy the hypothesis of Theorem \ref{t4}(c), i.e. $p,$ $p+2$ are twin primes.
{\renewcommand{\arraystretch}{1.1}
\begin{table}[ht]
\centering
\begin{tabular}{|c|c|c|c|}
\hline

$m$ & Twin primes $(p,p+2)$ & $h_{(m,p)}$ & $h_{(m,p+2)}$  \\
\hline 
$3$ & $(3,5)$   & $2\times3$ & $2^2\times 5$\\
\hline
$3$ & $(11,13)$ & $2^5 \times 11
$ &  $2^7 \times 13
$ \\
\hline

$15$ & $(5,7)$ & $2 \times 3 \times 5^2 \times 7
$ &  $2 \times 7 \times 1087
$ \\
\hline
$3$  & $(17,19)$ & $2^2 \times 3 \times 11 \times 17

$ & $2^2 \times 11 \times 13 \times 19
$\\
\hline
$35$ & $(5,7)$ &  $2^5 \times 5^2 \times 11
$ & $2^6 \times 7 \times 19 \times 47
$\\
\hline
$7$ & $(11,13)$ &  $2^4 \times 3 \times 11 \times 79
$     &   $2^5 \times 5 \times 13 \times 89
$\\
\hline
$29$ & $(5,7)$ & $2^5 \times 5^2
$ &  $
2^3 \times 3 \times 7 \times 523
$ \\
\hline 

\end{tabular}
\vbox{\vspace{0.35cm} Divisibility of class numbers by twin primes, $p\mid h_{(m,p)}$ or $(p+2)\mid h_{(m,p+2)}$.}
\end{table}
}
\newpage
\hspace{-1.35cm}Table 2: 
 Some examples which satisfy the hypothesis of Theorem \ref{t5}, i.e., If $k=p^r$ and $r>1$, then we can choose any odd $m\geq3$. If $k=p$, then we can choose any $m>2^{p-2}$.
{\renewcommand{\arraystretch}{2.0}
\begin{table}[ht]
\centering
\begin{tabular}{|c|c|c|}
\hline

$k$ & $m$ & $h_{(m,k)}$  \\
\hline
$5$ & $27(27>2^{5-2})$ & $2^3\times3^4\times5$ \\
\hline
$3^2$ & $15$ & $2^2 \times 3^4 \times 739
$  \\
\hline
$5^2$ &  $3$ & $2^4 \times 5^2 \times 43 \times 79
$ \\
\hline
$7$  &   $33(33>2^{7-2})$ & $2^7 \times 7 \times 199
$  \\
\hline
$3^2$ &  $11$ & $2^3 \times 3^2 \times 659
$ \\
\hline
\end{tabular}
\vbox{\vspace{0.3cm}The divisibility of class numbers $h_{(m,k)}$ by $k$.}
\end{table}
}

\hspace{-1.3cm}Table 3 : Examples which satisfy the hypothesis of Theorem \ref{t4}(b), i.e. $1-2m^p$ is squarefree.
{\renewcommand{\arraystretch}{2.2}
\begin{table}[ht]
\centering
\begin{tabular}{|c|c|c|c|}
\hline
$p$ & $m$ & $2m^p-1$ & $h_{(m,p)}$  \\
\hline
$3$ & $11$ &  $3\times887$ &  $2^4 \times 3
$\\
\hline
$5$ & $17$ &  $3\times37\times25583$ & $2^5 \times 5^2
$\\
\hline
$7$ & $21$ & $3602177081$ & $2^2 \times 3 \times 7 \times 1489
$ \\
\hline
$11$ & $9$ & $7\times439 \times607 \times33647$ & $2^9 \times 3^3 \times 11
$\\
\hline
$13$ & $7$ & $23 \times 121697 \times 1816247$ & $2^5 \times 5 \times 13 \times 89$\\
\hline
\end{tabular}
\vbox{\vspace{0.35cm}The divisibility of class numbers $h_{(m,p)}$ by $p$ when $1-2m^p$ is square-free.}
\end{table}
}

 \section*{Acknowledgement}
	We thank IISER, TVM for providing the excellent working conditions. We would like to thank Azizul Hoque and Sunil Kumar Pasupulati for helpful discussions and suggestions. We also would to thank Jaitra Chattopadhyay, Subham Bhakta and Jayanta Manoharmayum for helpful comments. The first author's research was supported by SERB grant CRG/2023/009035. The second author wishes to thank CSIR for financial support. We used Sage Math for calculations, hence we thank Sage Math for this. \\
	\\
	Declarations \\
Conflict of interest. The authors declare that there is no conflict of interest.
	
\end{document}